\documentclass[12pt,a4paper]{article}
\usepackage{mathrsfs}
\usepackage{indentfirst}
\usepackage{amsmath,amssymb,amsfonts,amsthm,graphics}
\usepackage{makeidx}
\usepackage{color}

\newtheorem{theorem}{Theorem}

\newtheorem{lemma}{Lemma}

\newtheorem{conj}{Conjecture}

\title{\bf Solutions to conjectures on the\\ $(k,\ell)$-rainbow index of
complete graphs\footnote{Supported by NSFC No. 11071130 and the ``973'' project.}}
\author{\small Qingqiong Cai, Xueliang Li, Jiangli Song\\
{\small Center for Combinatorics and LPMC-TJKLC}\\ {\small Nankai
University}\\ {\small Tianjin 300071, China}\\ {\small Email:
cqqnjnu620@163.com, lxl@nankai.edu.cn,
songjiangli@mail.nankai.edu.cn}}
\date{}

\begin{document}
\maketitle

\begin{abstract}
The $(k,\ell)$-rainbow index $rx_{k, \ell}(G)$ of a graph $G$ was
introduced by Chartrand et. al. For the complete graph $K_n$ of
order $n\geq 6$, they showed that $rx_{3,\ell}(K_n)=3$ for
$\ell=1,2$. Furthermore, they conjectured that for every positive
integer $\ell$, there exists a positive integer $N$ such that
$rx_{3,\ell}(K_{n})=3$ for every integer $n \geq N$. More generally,
they conjectured that for every pair of positive integers $k$ and
$\ell$ with $k\geq 3$, there exists a positive integer $N$ such that
$rx_{k,\ell}(K_{n})=k$ for every integer $n \geq N$. This paper is
to give solutions to these conjectures.
\end{abstract}

\noindent{\bf Keywords:} rainbow connectivity; rainbow tree; rainbow
index.

\noindent{\bf AMS subject classification 2010:} 05C40, 05C05, 05C15, 05D40.

\section {\large Introduction}

All graphs in this paper are undirected, finite and simple. We
follow \cite{Bondy} for graph theoretical notation and terminology
not described here. Let $G$ be a nontrivial connected graph with an
\emph{edge-coloring} $c: E(G)\rightarrow\{1, 2,\cdots, t\}, t \in
\mathbb{N}$, where adjacent edges may be colored the same. A path is
said to be a \emph{rainbow path} if no two edges on the path have
the same color. An edge-colored graph $G$ is called \emph{rainbow
connected} if for every pair of distinct vertices of $G$ there
exists a rainbow path connecting them. The \emph{rainbow connection
number} of a graph $G$, denoted by $rc(G)$, is defined as the
minimum number of colors that are needed in order to make $G$
rainbow connected. The \emph{rainbow $k$-connectivity} of $G$,
denoted by $rc_{k}(G)$, is defined as the minimum number of colors
in an edge-coloring of $G$ such that every two distinct vertices of
$G$ are connected by $k$ internally disjoint rainbow paths. These
concepts were introduced by Chartrand et. al. in \cite{ChartrandGL}.
Recently, there have been published a lot of results on the rainbow
connections. We refer the readers to \cite{LiSun} \cite{LiSun3} for
details.

Similarly, a tree $T$ in $G$ is called a \emph{rainbow tree} if no
two edges of $T$ have the same color. For $S\subseteq V(G)$, a
\emph{rainbow $S$-tree} is a rainbow tree connecting the vertices of
$S$. Suppose that $\{T_{1},T_{2},\cdots, T_{\ell}\}$ is a set of
rainbow $S$-trees. They are called \emph{internally disjoint} if
$E(T_{i})\cap E(T_{j})=\emptyset$ and $V(T_{i})\bigcap V(T_{j})=S$
for every pair of distinct integers $i,j$ with $1\leq i,j\leq \ell$
(Note that the trees are vertex-disjoint in $G\setminus S$). Given
two positive integers $k$, $\ell$ with $k\geq 2$, the
\emph{$(k,\ell)$-rainbow index} $rx_{k,\ell}(G)$ of $G$ is the
minimum number of colors needed in an edge-coloring of $G$ such that
for any set $S$ of $k$ vertices of $G$, there exist $\ell$
internally disjoint rainbow $S$-trees. In particular, for $\ell=1$,
we often write $rx_{k}(G)$ rather than $rx_{k,1}(G)$ and call it the
\emph{$k$-rainbow index}. It is easy to see that
$rx_{2,\ell}(G)=rc_{\ell}(G)$. So the $(k,\ell)$-rainbow index can
be viewed as a generalization of the rainbow connectivity. In the
sequel, we always assume $k\geq 3$. The concept of
$(k,\ell)$-rainbow index was also introduced by Chartrand et. al. in
\cite{Chartrand}. They determined the $k$-rainbow index of all
unicyclic graphs and the $(3,\ell)$-rainbow index of complete graphs
for $\ell=1,2$. In the end of \cite{Chartrand}, they proposed the
following two conjectures:

\begin{conj}\label{conj1}
For every positive integer $\ell$, there exists a positive integer $N$
such that $rx_{3,\ell}(K_{n})=3$ for every integer $n \geq N$.
\end{conj}

\begin{conj}\label{conj2}
For every pair of positive integers  $k,\ell$ with $k\geq 3$, there
exists a positive integer $N$ such that $rx_{k,\ell}(K_{n})=k$ for
every integer $n \geq N$.
\end{conj}

In this paper, we will apply the probabilistic method \cite{Alon} to
solve the above two conjectures.

\section{Solution to the conjectures}

It is easy to see that the second conjecture implies the first one.
So, if the second conjecture is solved, the first one follows then.
In this section, we will solve Conjecture \ref{conj2}. Firstly, let
us start with a lemma.

\begin{lemma}\label{1}
For every pair of positive integers $k,\ell$ with $k\geq3$, there
exists a positive integer
$N_{1}=4\lceil(\frac{k+\ell-1}{\ln(1-k!/k^{k})})^{2}\rceil$ such
that $rx_{k,\ell}(K_{n})\leq k$ for every integer $n \geq N_{1}$.
\end{lemma}

\begin{proof}
Let $C=\{1,2,\cdots,k\}$ be a set of $k$ different colors. We color
the edges of $K_{n}$ with the colors from $C$ randomly and
independently. For $S\subseteq V(K_{n})$ with $|S|=k$, define
$A_{S}$ as the event that there exist at least $\ell$ internally
disjoint rainbow $S$-trees. If Pr[ $\bigcap\limits_{S}A_{S}$ ]$>0$,
then there exists a suitable $k$-edge-coloring, which implies that
$rx_{k,\ell}(K_{n})\leq k$.

Let $S\subseteq V(K_{n})$ with $|S|=k$. Without loss of generality,
we suppose $S=\{v_{1}, v_{2}, \cdots, v_{k}\}$. For any vertex $u
\in V(K_{n}) \setminus S$, let $T(u)$ denote a star with $u$ as its
center and $E(T(u))=\{uv_{1}, uv_{2},\cdots, uv_{k}\}$. Clearly,
$T(u)$ is an $S$-tree. Moreover, for $u_{1}, u_{2} \in V(K_{n})
\setminus S$ and $u_{1}\neq u_{2}$, $T(u_{1})$ and $T(u_{2})$ are
two internally disjoint $S$-trees. Let $\mathcal{T^{*}}=\{T(u)|u \in
V(K_{n}) \setminus S\}$. Then $\mathcal{T^{*}}$ is a set of $n-k$
internally disjoint $S$-trees. It is easy to see that $p$:= Pr[ $T
\in \mathcal{T^{*}}$ is a rainbow $S$-tree ]$=k!/k^{k}$ (Throughout
this paper, $\mathcal{T^{*}}$ and $p$ are always defined as this).
Denote by $B_{S}$ the event that there exist at most $\ell-1$
internally disjoint rainbow $S$-trees in $\mathcal{T^{*}}$. Here we
assume that $n\geq k+\ell\geq4$. Then $n-k > \ell-1$ and
\begin{center}
Pr[ $\overline{A_{S}}$ ]$\leq$ Pr[ $B_{S}$ ]$\leq {n-k \choose
\ell-1}(1-p)^{n-k-(\ell-1)}$$<n^{\ell-1}(1-p)^{n-k-\ell+1}$.
\end{center}
As an immediate consequence, we get that
\begin{eqnarray*}
  Pr[\ \bigcap\limits_{S} A_{S}\ ] &=& 1- Pr[\ \bigcup\limits_{S} \overline{A_{S}}\ ] \\
   &\geq& 1-\sum\limits_{S}Pr[\ \overline{A_{S}}\ ] \\
   &>& 1-\sum\limits_{S}n^{\ell-1}(1-p)^{n-k-\ell+1} \\
   &=& 1-{n \choose k}n^{\ell-1}(1-p)^{n-(k+\ell-1)} \\
   &>& 1- n^{k+\ell-1}(1-p)^{n-(k+\ell-1)}.
\end{eqnarray*}

Now we are in the position to estimate the value of $N_{1}$
according to the inequality $n^{k+\ell-1}(1-p)^{n-(k+\ell-1)}\leq1$,
which leads to Pr[ $\bigcap\limits_{S}A_{S}$ ]$>0$. This
inequality is equivalent to
\begin{center}
$(\frac{n}{1-p})^{k+\ell-1} \leq (\frac{1}{1-p})^{n}$.
\end{center}
Taking the natural logarithm, we get that
\begin{center}
$(k+\ell-1)\ln \frac{n}{1-p} \leq n\ln \frac{1}{1-p}$.
\end{center}
That is,
\begin{center}
 $\frac{k+\ell-1}{\ln (1/(1-p))}\leq \frac{n}{\ln n+\ln (1/(1-p))}$.
\end{center}

Let $f(k)=\frac{1}{1-p}=\frac{1}{1-k!/k^{k}}$. Obviously, $f(k)$ is
monotonically decreasing in $[3,+\infty)$. So, $f(k)\leq f(3)\approx
1.286$. Since $n \geq 4 > \frac{1}{1-p}$, $\ln n > \ln
\frac{1}{1-p}$, then $\frac{n}{\ln n+\ln (1/(1-p))}>\frac{n}{2\ln
n}$. Note that $\ln x < \sqrt{x}$ holds for $x\geq 4$. Thus, when
$n\geq k+\ell\geq4$, we have $\frac{n}{\ln n+\ln
(1/(1-p))}>\frac{\sqrt{n}}{2}$. Setting $\frac{k+\ell-1}{\ln
(1/(1-p))}\leq\frac{\sqrt{n}}{2}$, we get that
$n\geq4(\frac{k+\ell-1}{\ln (1/(1-p))})^{2}$. Then, the inequality
$\frac{k+\ell-1}{\ln (1/(1-p))}< \frac{n}{\ln n+\ln (1/(1-p))}$
holds for $n\geq\max\{k+\ell,4(\frac{k+\ell-1}{\ln
(1/(1-p))})^{2}\}=4(\frac{k+l-1}{\ln (1/(1-p))})^{2}$. In other
words, if $n\geq N_{1}=4\lceil(\frac{k+\ell-1}{\ln
{(1-k!/k^{k}})})^{2}\rceil$, then $Pr[\ \bigcap\limits_{S} A_{S}\
]>0$, as desired.
\end{proof}

To solve Conjecture \ref{conj2} completely, we have to determine an
integer $N_{2}$ such that for every integer $n \geq N_{2}$,
$rx_{k,\ell}(K_{n})\geq k$. First we recall the concept of Ramsey
number, which will be used in our proof. The Ramsey number $R(t,s)$
is the smallest integer $n$ such that every 2-edge-coloring of
$K_{n}$ contains either a complete subgraph on $t$ vertices, all of
whose edges are assigned color 1, or a complete subgraph on $s$
vertices, all of whose edges are assigned color 2. For positive
integers $t_{i}$ with $1\leq i\leq r$, the multicolor Ramsey number
$R(t_{1},t_{2},\cdots,t_{r})$ is defined as the smallest integer $n$
such that for every $r$-edge-coloring of $K_{n}$, there exists an
$i\in\{1,2,\cdots,r\}$ such that $K_n$ contains a complete subgraph
on $t_{i}$ vertices, all of whose edges are assigned color $i$. When
$t_{1}=t_{2}=\cdots=t_{r}=t$, $R(t_{1},t_{2},\cdots,t_{r})$ is
abbreviated to $R_{r}(t)$. The existence of such a positive integer
is guaranteed by the Ramsey's classical result \cite{Ramsey}. A
survey on the Ramsey number of graphs can be found in
\cite{Radziszowski}. A typical upper bound for the multicolor Ramsey
number is as follows, which can be found in any related textbooks,
see \cite{Bondy} for example. For all positive integers $t_{i}$ with
$1\leq i \leq r$,
\begin{center}
$R(t_{1}+1,t_{2}+1,\cdots,t_{r}+1)\leq
\frac{(t_{1}+t_{2}+\cdots+t_{r})!}{t_{1}!t_{2}!\cdots t_{r}!}$. \ \
\ \ \ \ \ \ \ \ \ \ \ \ \ \ \ \ \ \ \ (1)
\end{center}
One may find more refined upper bounds in the existing literature,
see \cite{Radziszowski} for example.

For $S\subseteq V(G)$ with $|S|=k$, let $\mathcal{T}$ be a maximum
set of internally disjoint rainbow $S$-trees in $G$. Let
$\mathcal{T}_1$ be the set of rainbow $S$-trees in $\mathcal{T}$, all of
whose edges belong to $E(G[S])$, and $\mathcal{T}_2$ be the set of
rainbow $S$-trees in $\mathcal{T}$ containing at least one edge from
$E_G[S,\overline{S}]$. Clearly, $\mathcal{T}=\mathcal{T}_1\cup
\mathcal{T}_2$ (Throughout this paper, $\mathcal{T}$,
$\mathcal{T}_1$, $\mathcal{T}_2$ are always defined as this).

\begin{lemma}\label{lem}
For $S\subseteq V(G)$ with $|S|=k$, let $T$ be a rainbow $S$-tree.
If $T\in \mathcal{T}_1$, then $T$ uses exactly $k-1$ different
colors; if $T\in \mathcal{T}_2$, then $T$ uses at least $k$
different colors.
\end{lemma}

\begin{proof}
It is easy to see that for each rainbow $S$-tree
$T\in\mathcal{T}_1$, $T$ has exactly $k-1$ edges. Then, exactly
$k-1$ different colors are used. For each rainbow $S$-tree
$T\in\mathcal{T}_2$, $T$ contains at least one vertex in
$V(G)\setminus S$. Then, $T$ has at least $k+1$ vertices. So the
number of edges of $T$ is at least $k$, which implies that $T$ uses
at least $k$ different colors.
\end{proof}

We proceed with the following lemma.

\begin{lemma}\label{2}
For every pair of positive integers $k,\ell$ with $k\geq3$,

(i) if $\ell > \lfloor\frac{k}{2}\rfloor$, then
$rx_{k,\ell}(K_{n})\geq k$ for every integer $n \geq  N_{2}=k$;

(ii) if $\ell \leq \lfloor\frac{k}{2}\rfloor$, then $rx_{k,\ell}(K_{n})\geq k$
for every integer $n \geq N_{2}= R_{k-1}(k)$.

\end{lemma}

\begin{proof} We distinguish two cases.

\emph{Case 1. $\ell>\lfloor\frac{k}{2}\rfloor$}

For any set $S$ of $k$ vertices in $K_{n}$, the induced subgraph by
$S$, denoted by $G[S]$, is a complete graph of order $k$. So, by
Theorem 3.1 of \cite{Chartrand} we know that $G[S]$ contains at most
$\lfloor\frac{k}{2}\rfloor$ edge-disjoint spanning trees. From
$\ell>\lfloor\frac{k}{2}\rfloor$, we can derive that there must
exist one rainbow $S$-tree in $\mathcal{T}_2$, which uses at least
$k$ different colors by Lemma \ref{lem}. Thus
$rx_{k,\ell}(K_{n})\geq k$ for every integer $n\geq k$.

\emph{Case 2. $\ell\leq\lfloor\frac{k}{2}\rfloor$}

From the Ramsey's theorem, we know that if $k\geq 3$ and $n\geq
R_{k-1}(k)$, then in any $(k-1)$-edge-coloring of $K_{n}$, one will
find a monochromatic subgraph $K_{k}$. Now, take $S$ as the set of
$k$ vertices of the monochromatic subgraph $K_{k}$. Then,
$\mathcal{T}_1=\emptyset$. In other words, all the rainbow $S$-trees
belong to $\mathcal{T}_2$. Similar to \emph{Case 1}, we get that
$rx_{k,\ell}(K_{n})\geq k$ for every integer $n \geq
R_{k-1}(k)$.
\end{proof}

Combining Lemmas \ref{1} and \ref{2}, we come to the following
conclusion, which solves Conjecture \ref{conj2}.

\begin{theorem}\label{thm}
For every pair of positive integers $k,\ell$ with $k\geq3$,

(i) if $\ell > \lfloor\frac{k}{2}\rfloor$, then there exists a positive
integer $N=4\lceil(\frac{k+\ell-1}{\ln(1-k!/k^{k})})^{2}\rceil$ such
that $rx_{k,\ell}(K_{n})=k$ for every integer $n \geq N$.

(ii) if $\ell \leq \lfloor\frac{k}{2}\rfloor$, there exists a positive
integer $N=\max\{4\lceil(\frac{k+\ell-1}{\ln(1-k!/k^{k})})^{2}\rceil,R_{k-1}(k)\}$
such that $rx_{k,\ell}(K_{n})=k$ for every integer $n \geq N$.
\end{theorem}

Note that although this gives a lower bound $N$ for the order $n$ of
a complete graph with $rx_{k,\ell}(K_{n})=k$, the bound is far from
the best. Also, note that from Inequ.(1) we can get a rough upper
bound for the Ramsey number $R_{k-1}(k)\leq
\frac{((k-1)^{2})!}{((k-1)!)^{k-1}}$. Next section we will use this
bound to investigate a more exact solution of $N$ for the
($3,\ell$)-rainbow index $rx_{3,\ell}(K_{n})$.

\section{Exact asymptotic solution of $N$ for $k=3$}

In this section, we will focus on the exact asymptotic solution of
$N$ for the $(3,\ell)$-rainbow index of $K_{n}$. To start with, we
present a result derived from Theorem \ref{thm}.
\begin{lemma}\label{lem2}
For every positive integers $\ell$, there exists an integer
$N=4\lceil(\frac{\ell+2}{\ln9/7})^{2}\rceil$ such that $rx_{3,\ell}(K_{n})=3$
for every integer $n\geq N$.
\end{lemma}

\begin{proof}
From Lemma \ref{1}, we know that $rx_{3,\ell}(K_{n})\leq 3$ for
every integer $n\geq 4\lceil(\frac{\ell+2}{\ln9/7})^{2}\rceil$. On
the other hand, it follows from Lemma \ref{2} that
$rx_{3,\ell}(K_{n})\geq3$ for every integer $n\geq 6$. Since
$4\lceil(\frac{\ell+2}{\ln9/7})^{2}\rceil>6 $ holds for all integers
$\ell\geq1$, $rx_{3,\ell}(K_{n})=3$ for every integer $n\geq
4\lceil(\frac{\ell+2}{\ln9/7})^{2}\rceil$.
\end{proof}

One can see that the value of $N$ in Lemma \ref{lem2} is
$O(\ell^{2})$, which is far from the best. Next step, we will
improve $N$ to $\frac{9}{2}\ell+o(\ell)$ in a certain range for
$\ell$, and show that it is  asymptotically the best possible. To
see this, we start with a general lemma for all integers $k\geq3$.

\begin{lemma}\label{3}
Let $\varepsilon$ be a constant with $0< \varepsilon <1$, $k,\ell$
be two integers with $k\geq 3$ and $\ell\geq
\frac{k!}{k^{k}}(\theta-k)(1-\varepsilon)+1$, where
$\theta=\theta(\varepsilon,k)$ is the largest solution of
$x^{k}e^{-\frac{k!}{2k^{k}}\varepsilon^{2}(x-k)}=1$. Then,
$rx_{k,\ell}(K_{n})\leq k$ for every integer $n\geq
\lceil\frac{k^{k}(\ell-1)}{k!(1-\varepsilon)}+k\rceil$.
\end{lemma}

\begin{proof}
Here we follow the notations $C,S,A_{S},T(u),p,\mathcal{T^{*}}$ in
the proof of Lemma \ref{1}. Color the edges of $K_{n}$ with the
colors from $C$ randomly and independently. Just like in Lemma
\ref{1}, our aim is to obtain Pr[ $\bigcap\limits_{S}A_{S}$ ]$>0$.
We assume $n>k$.

Let $X$ be the number of rainbow $S$-trees in $\mathcal{T^{*}}$.
Clearly, $X\sim Bi(n-k,p)$ and $EX=(n-k)p$. Using the Chernoff Bound
\cite{Alon}, we get that $$Pr[\ \overline{A_{S}}\ ]\leq Pr[\ X\leq
\ell-1\ ]=Pr[\ X\leq (n-k)p(1- \frac{(n-k)p-\ell+1}{(n-k)p})\ ]\leq
e^{-\frac{1}{2}[\frac{(n-k)p-\ell+1}{(n-k)p}]^{2}p(n-k)}.$$
Note
that the condition $n\geq \frac{ \ell-1}{p(1-\varepsilon)}+k$
ensures $(n-k)p>\ell-1$. So we can apply the Chernoff Bound to
scaling the above inequalities. Also since $n\geq \frac{
\ell-1}{p(1-\varepsilon)}+k$, then
$\frac{(n-k)p-\ell+1}{(n-k)p}\geq\varepsilon$, and thus $Pr[\
\overline{A_{S}}\ ]\leq e^{-\frac{1}{2}\varepsilon^{2}p(n-k)}$. So,
\begin{eqnarray*}
  Pr[\ \bigcap\limits_{S} A_{S}\ ] &=& 1- Pr[\ \bigcup\limits_{S} \overline{A_{S}}\ ] \\
  &\geq& 1-\sum\limits_{S}Pr[\ \overline{A_{S}}\ ]\\
  &\geq& 1-\sum\limits_{S}e^{-\frac{1}{2}\varepsilon^{2}p(n-k)}\\
   &=& 1-{n \choose k}e^{-\frac{1}{2}\varepsilon^{2}p(n-k)} \\
   &>& 1-n^{k}e^{-\frac{1}{2}\varepsilon^{2}p(n-k)}.\\
\end{eqnarray*}
Obviously, the function
$f(x)=x^{k}e^{-\frac{1}{2}\varepsilon^{2}p(x-k)}$ eventually
decreases and tends to $0$ as $x\rightarrow +\infty$. Let
$\theta=\theta(\varepsilon,k)$ be the largest solution of
$x^{k}e^{-\frac{1}{2}\varepsilon^{2}p(x-k)}=1$. Then, if
$n\geq\theta$, then
$n^{k}e^{-\frac{1}{2}\varepsilon^{2}p(n-k)}\leq1$, and consequently,
$Pr[\ \bigcap\limits_{S} A_{S}\ ]>0$, as desired. On the other hand,
since $\ell\geq p(\theta-k)(1-\varepsilon)+1$, then $n\geq \frac{
(\ell-1)}{p(1-\varepsilon)}+k\geq \theta$, which completes our
proof.
\end{proof}

Let $k=3$. From Lemma \ref{3} we know that if $0< \varepsilon <1$,
$\ell$ is an integer with $\ell\geq
\frac{2}{9}(\theta-3)(1-\varepsilon)+1$ where
$\theta=\theta(\varepsilon)$ is the largest solution of
$x^{3}e^{-\frac{1}{9}\varepsilon^{2}(x-3)}=1$, then
$rx_{3,\ell}(K_{n})\leq 3$ for every integer $n\geq
\lceil\frac{9(\ell-1)}{2(1-\varepsilon)}+3\rceil$. On the other
hand, it follows from Lemma \ref{2} that $rx_{3,\ell}(K_{n})\geq 3$
for all integers $n\geq6$. Thus we get the following theorem.

\begin{theorem}
Let $\varepsilon$ be a constant with $0< \varepsilon <1$, $\ell$ be
an integer with $\ell\geq \frac{2}{9}(\theta-3)(1-\varepsilon)+1$
where $\theta=\theta(\varepsilon)$ is the largest solution of
$x^{3}e^{-\frac{1}{9}\varepsilon^{2}(x-3)}=1$. Then, there exists an
integer
$N=\max\{6,\lceil\frac{9(\ell-1)}{2(1-\varepsilon)}+3\rceil\}$ such
that $rx_{3,\ell}(K_{n})= 3$ for every integer $n\geq N$.
\end{theorem}

For example, if we set $\varepsilon=\frac{1}{2}$, then
$\theta\approx712.415$. The result shows that for $\ell\geq80$,
$rx_{3,\ell}(K_{n})= 3$ holds for every integer $n\geq 9\ell-6$. If
we set $\varepsilon=\frac{2}{3}$, then $\theta\approx360.699$. The
result shows that for $\ell\geq28$, $rx_{3,\ell}(K_{n})= 3$ holds
for every integer $n\geq \frac{3}{2}(9\ell-7)$.

Now we have improved $N$ from $O(\ell^{2})$ to
$\frac{9}{2}\ell+o(\ell)$. A natural question is how small the
integer $N$ can be. The next lemma will show that
$\frac{9}{2}\ell+o(\ell)$ is asymptotically the best possible.

\begin{lemma} \label{lem3} For any $3$-edge-coloring of $K_{n}$,
there exists a set $S\subseteq V(K_{n})$ with $|S|=3$ such that the
number of internally disjoint rainbow $S$-trees is at most
$\frac{2(n-1)^{2}}{9(n-2)}+3$.
\end{lemma}

\begin{proof}
Let $C$ be an arbitrary $3$-edge-coloring of $K_{n}$. For every set
$S\subseteq V(K_{n})$ with $|S|=3$, we define the following three
variables:

$\bullet$ $X(S)$ is the number of internally disjoint rainbow
$S$-trees;

$\bullet$ $X_{1}(S)$ is the number of internally disjoint rainbow
$S$-trees that contains at least one edge in $E(G[S])$;

$\bullet$ $X_{2}(S)$ is the number of internally disjoint rainbow
$S$-trees in $\mathcal{T^{*}}=\{T(u)|u \in V(K_{n}) \setminus S\}.$

In fact, $X(S)=X_{1}(S)+X_{2}(S)$. Moreover, $X_{1}(S)\leq 3$ since
there are exactly three edges in $E(G[S])$.

For any vertex $v\in V(K_{n})$, we define $Y_{v}$ as the number of
distinct rainbow stars with $3$ edges and with $v$ as its center.
Denote by $d_{i}(v)$ $(1\leq i\leq 3)$ the number of edges of color
$i$ incident with $v$. Apparently,
$d_{1}(v)+d_{2}(v)+d_{3}(v)=d(v)=n-1$. Counting the distinct
rainbow stars in two ways, we have
$\sum\limits_{S}{X_{2}(S)}=\sum\limits_{v}{Y_{v}}$. Then
\begin{eqnarray*}
  EX &=& \frac{1}{ {n\choose 3}} \sum\limits_{S}{X(S)}\\
  &=& \frac{1}{{n \choose 3}} (\sum\limits_{S}{X_{1}(S)}+\sum\limits_{S}{X_{2}(S)})\\
  &\leq& \frac{1}{{n \choose 3}} (\sum\limits_{S}{3}+\sum\limits_{v}{Y_{v}})\\
  &=& 3+ \frac{1}{{n \choose
  3}}\sum\limits_{v}{d_{1}(v)d_{2}(v)d_{3}(v)}\\
  &\leq& 3+ \frac{1}{{n \choose 3}}\sum\limits_{v}{(\frac{d_{1}(v)+d_{2}(v)+d_{3}(v)}{3})^{3}}    \\
     &=&3+\frac{1}{{n \choose 3}}\sum\limits_{v}{(\frac{n-1}{3})^{3}} \\
     &=&3+\frac{n}{{n \choose 3}}(\frac{n-1}{3})^{3} \\
     &=&3+\frac{2(n-1)^{2}}{9(n-2)}.
\end{eqnarray*}

Therefore, there exists a set $S$ of three vertices such that the
number of internally disjoint rainbow $S$-trees is at most
$\frac{2(n-1)^{2}}{9(n-2)}+3$.
\end{proof}

It follows from the above lemma that
$\ell\leq\frac{2(n-1)^{2}}{9(n-2)}+3$, which is approximately
equivalent to $n\geq\frac{9}{2}\ell+o(\ell)$. Therefore,
$\frac{9}{2}\ell+o(\ell)$ is asymptotically the best possible for
the lower bound on $N$.

\section{Concluding remark}

In this paper, we solve the two conjectures in \cite{Chartrand}. At
first we prove that for every pair of positive integers $k,\ell$
with $k\geq3$, if $n \geq
4\lceil(\frac{k+\ell-1}{\ln(1-k!/k^{k})})^{2}\rceil$, then
$rx_{k,\ell}(K_{n})\leq k$. Recall that the Ramsey number
$R_{k-1}(k)$ is the smallest number $n$ such that any
$(k-1)$-edge-coloring of $K_{n}$ yields a monochromatic subgraph
$K_{k}$. So, if $n\geq R_{k-1}(k)$, then $rx_{k,\ell}(K_{n})\geq k$
(Note that $R_{k-1}(k)\leq \frac{((k-1)^{2})!}{((k-1)!)^{k-1}}$).
Thus, we get that $rx_{k,\ell}(K_{n})=k$ for every integer $n\geq
N=\max\{4\lceil(\frac{k+\ell-1}{\ln(1-k!/k^{k})})^{2}\rceil,R_{k-1}(k)\}$,
which solves Conjecture \ref{conj2}. Then, we try to get a more
exact asymptotic solution of $N$ for the special case $k=3$. Using
the Chernoff Bound, we obtain that if $n\geq N=\max
\{6,\lceil\frac{9(\ell-1)}{2(1-\varepsilon)}+3\rceil$\}, where
$0<\varepsilon<1$, then $rx_{3,\ell}(K_{n})=3$; moreover the bound
$\frac{9}{2}\ell+o(\ell)$ is asymptotically the best possible for
$N$ in Conjecture \ref{conj1}.


\begin{thebibliography}{99}

\bibitem{Alon} N. Alon, J.H. Spencer, \emph{The Probabilistic Method}, John Wiley \& Sons, 2004.

\bibitem{Bondy} J. Bondy, U.S.R. Murty, \emph{Graph Theory}, GTM 244, Springer, 2008.

\bibitem{ChartrandGL} G. Chartrand, G. Johns, K. McKeon, P. Zhang,
\emph{ The rainbow connectivity of a graph}, Networks 1002(2009), 75-81.

\bibitem{Chartrand} G. Chartrand, F. Okamoto, P. Zhang,
\emph{ Rainbow trees in graphs and generalized connectivity},
Networks 55(2010), 360-367.

\bibitem{LiSun3} X. Li, Y. Shi, Y. Sun, \emph{ Rainbow connections of graphs: A Survey}, Graphs \& Combin. 29(2013), 1-38.

\bibitem{LiSun} X. Li, Y. Sun, \emph{Rainbow Connections of Graphs}, SpringerBriefs in Math. Springer, New York, 2012.

\bibitem{Radziszowski} S. Radziszowski, \emph{ Small Ramsey numbers }, Electron. J. Combin. 1(1994),
Dynamic Survey, DS1.12 (August 4, 2009).

\bibitem{Ramsey} F. Ramsey, \emph{ On a problem of formal logic,} Proc. London Math. Soc. 2nd Ser. 30(1930), 264-286.

\end{thebibliography}
\end{document}